\documentclass[11pt]{amsart}

\usepackage{amsmath,amsfonts,amssymb,amsthm,enumerate,tikz-cd}
\newcommand{\nc}{\newcommand}
\def\Z{\mathbb Z}
\def\Q{\mathbb Q}

\newcommand{\dR}{\mathrm{dR}}

\newcommand{\Sym}{\mathrm{Sym}}

\newcommand{\an}{\mathrm{an}}
\newcommand{\codim}{\mathrm{codim}}
\newcommand{\tors}{\mathrm{tors}}

\newcommand{\rk}{\mathrm{rk}}
\newcommand{\Lie}{\mathrm{Lie}}
\newcommand{\Ker}{\mathrm{Ker}}

\newcommand{\val}{\mathrm{val}}
\newcommand{\End}{\mathrm{End}}

\theoremstyle{plain}
\newtheorem{theorem}{Theorem}

\newtheorem{lemma}{Lemma}
\newtheorem{proposition}{Proposition}
\newtheorem{corollary}{Corollary}

\theoremstyle{definition}

\newtheorem{definition}{Definition}

\newcommand{\Gal}{\mathrm{Gal}}
\newcommand{\nr}{\mathrm{nr}}
\newcommand{\F}{\mathbb{F}}
\newcommand{\Mat}{\mathrm{Mat}}

\newcommand{\red}{\mathrm{red}}
\newcommand{\Hom}{\mathrm{Hom}}
\newcommand{\Jac}{\mathrm{Jac}}

\newcommand{\cris}{\mathrm{cris}}

\newcommand{\nab}{\nabla_{0}^{1}}
\usepackage[all]{xy}
\DeclareMathOperator{\Spec}{\mathrm{Spec}}
\DeclareMathOperator{\Spf}{\mathrm{Spf}}

\newcommand{\sO}{\mathcal{O}}

\nc{\tR}{\tilde{R}}
\nc{\bx}{\mathbf{x}}
\nc{\by}{\mathbf{y}}
\nc{\bz}{\mathbf{z}}
\nc{\ba}{\mathbf{a}}
\nc{\Fp}{\tilde{F}}
\nc{\Rp}{\tilde{R}}
\nc{\mlow}{m_{\mathrm{l}}}
\nc{\mup}{m_{\mathrm{u}}}
\nc{\bXp}{\bX_{\mathrm{prim}}}
\nc{\bPsi}{\mathbf{\Psi}}
\nc{\mult}{\mathrm{mult}}
\nc{\mbB}{\mathbbm{B}}
\nc{\mfor}[1]{{#1}^{\mathrm{for}}}
\nc{\Hdr}{\mathbf{H}^1_{\mathrm{dR}}(A)}
\nc{\bt}{{\bf t}}
\nc{\beqar}{\begin{eqnarray*}}
\nc{\eeqar}{\end{eqnarray*}}
\nc{\fra}{\mfrak{f}}
\nc{\tW}{\tilde{W}}

\nc{\bo}{{\bf b}}
\nc{\hq}{{q}}

\nc{\mn}{[m]n}
\nc{\EE}{\mathbb{E}}
\nc{\PP}{\mathbb{P}}
\nc{\Fdr}{F_{\mathrm{dr}}}
\nc{\cM}{\mathcal{M}}
\nc{\fX}{\mathfrak{X}}
\nc{\fY}{\mathfrak{Y}}
\nc{\fA}{\mathfrak{A}}
\nc{\fN}{\mathfrak{N}}

\begin{document}
\title{A Buium--Coleman bound for the Zilber--Pink conjecture for curves inside abelian varieties}
\author[N. Dogra]{Netan Dogra}
\author[S. Pandit]{Sudip Pandit}
\maketitle

\begin{abstract} 
We provide an explicit bound for the image of the Zilber--Pink locus of a curve in an abelian variety under reduction modulo a suitably large prime $p$. We also provide a new proof of finiteness of the Zilber--Pink locus. Under further assumptions on the abelian variety, we prove a bound on the size of the locus.
\end{abstract}
\pagestyle{headings}
\markright{A  BUIUM--COLEMAN BOUND FOR THE ZILBER--PINK CONJECTURE}
\tableofcontents


\section{Introduction}
Let $p$ be a prime number. Let $R$ be the $p$-adic completion of the ring of integers of the maximal unramified extension $\Q _p ^{\mathrm{nr}}$ of $\Q _p $, and $k=\overline{\F }_p $ its residue field. Let $X$ be a smooth curve over $R$ of genus $g_X <p/2$ mapped into an abelian variety $A/R$ of dimension $g_A$. We suppose that $X$ generates $A$, which for us will mean that $\Jac (X)$ surjects onto $A$. In particular $g_A \leq g_X$.
Define 
\[
A^{[2]}:=\bigcup_{\mathrm{codim} (B)\geq 2} B(\overline{\Q }_p ),
\]
where the union is over all subgroups of $A$ of codimension at least two.  The Zilber--Pink conjecture, for $X\subset A$ as above, states that the set $X(\overline{\mathbb{\Q }}_p )\cap A^{[2]}$ is finite. This was first proved (for $X$ defined over a number field) by Habegger and Pila \cite{HP}. Special cases were proved earlier by R\'emond \cite{rem3} and Viada \cite{viada1} \cite{viada2}. When $X$ is not defined over a number field, it was proved by Barroero and Dill \cite{BD}. 

The Zilber--Pink conjecture makes sense in a far greater level of generality. We refer the reader to the original conjectures by Zilber \cite{zilber}, Bombieri--Masser--Zannier \cite{BMZ}, and Pink \cite{pink}, and to the surveys by Pila \cite{pila} and Zannier \cite{zannier}. An even more general conjecture for variations of mixed Hodge structure was formulated by Klingler \cite{klingler}.

In this paper, we provide a new proof of the Zilber--Pink conjecture for $X\subset A$ as above. 
If we impose further conditions, then we can make the proof quantitative. We say that $A$ satisfies property (*) if there is an isogeny $A_{\overline{K}}\sim \prod _{i=1}^{n_A} A_i ^{n_i }$, where the $A_i$ are pairwise non-isogenous simple abelian varieties over $\overline{K}$, and for all $i$, $n_i <\dim (A_i )/\rk (\End (A_i ))$. Let 
\[
\red :X(\overline{\Q }_p )\to X(k)
\]
denote the reduction map.
\begin{theorem}\label{theorem:main}
Let $H$ be an ample line bundle on $A_k$ and let $n_A$ be as above. 
Then we have the following.
\begin{enumerate}
\item 
\begin{align*}
& \# \red (X(R)\cap A^{[2]}) \\
\leq & \left(1+\sum _{e=1}^{g_A-2}\frac{p^{2\binom{g_A}{e}}-1}{p-1}\right)p^{g_A} (\deg _H X_k +p(2g_X -2)) \binom{2g_A -2}{g_A }(6p^2 )^{g_A }\deg _H A_k
\end{align*}
\item We have 
\[
\# \red ((X(\overline{\Q }_p )-X(R))\cap A^{[2]}) \leq (2g_X -2)\left(1+\sum _{e=1}^{g_A -2}\frac{p^{2\binom{g_A }{e}}-1}{p-1}\right).
\]
\item If $A$ satisfies property (*), then
\[
\# X(\overline{\Q }_p )\cap A^{[2]}\leq c\cdot \left( r_1 \cdot \left( 1+\frac{1}{(1 -\frac{1}{p-1})\log (p)}\right)+r_2 \cdot \left( 1+\frac{1}{(\frac{1}{2g_X} -\frac{1}{p-1})\log (p)}\right) \right) 
\]
where $r_1 := \# \red (X(R)\cap A^{[2]}), r_2 := \# \red ((X(\overline{\Q }_p )-X(R))\cap A^{[2]}) $, and 
\[
c:=(4g_A )^{\dim (A)}\cdot n_A \cdot  \prod _{i=1}^{\dim (A)} \binom{i+g_X +1}{g_X +1},
\]
\end{enumerate}
\end{theorem}
This result is inspired by the work of Buium and Coleman on explicit Manin--Mumford bounds in the unramified case \cite{buium:96} \cite{coleman:ramified}, and we follow their approach, with a few changes. We also use several ideas from the related (earlier) proof of Manin--Mumford by Raynaud \cite{raynaud}.

If $A$ satisfies property (*), then Theorem \ref{theorem:main} gives a new, quantitative proof of this case of Zilber--Pink. When $A$ does not satisfy property (*), Theorem \ref{theorem:main} has the following weaker consequence.
\begin{corollary}\label{cor:deg}
Let $\iota :X\to A$ be a smooth curve mapped into an abelian variety, such that $\Jac (X)$ surjects onto $A$ and $\iota ,X$ and $A$ are defined over a number field $K$. Then the set $X(\overline{K})\cap A^{[2]}$ is of bounded degree, i.e. there is a positive integer $d$ such that for all $x$ in this set, $[K(x):K]\leq d$.
\end{corollary}
By the Northcott property, this gives a new (non-quantitative) proof of Zilber--Pink when combined with the following theorem, proved independently by Habegger and R\'emond \cite{Hab} \cite{rem1}.
\begin{theorem}[Habegger, R\'emond]
For $X\subset A$ as in Corollary \ref{cor:deg}, $X(\overline{\Q })\cap A^{[2]}$ is a set of bounded height, i.e. there is an $N>0$ such that, for all $z\in X(\overline{\Q })\cap A^{[2]}$, $h(z)<N$.
\end{theorem}

\subsection{The Buium--Coleman approach to unlikely intersection problems}
The strategy of proof of Theorem \ref{theorem:main} is an extension of the proof of Manin--Mumford for curves in their Jacobians due to Buium and Coleman \cite{buium:96} \cite{coleman:ramified}, and also incorporates ideas from the original proof of Manin--Mumford for curves in abelian varieties by Raynaud \cite{raynaud}.

Recall that the Manin--Mumford conjecture (for $X\subset A$ as above) states that $X\cap A_{\tors }(\overline{\Q }_p )$ is finite. The Buium--Coleman approach to this theorem (in the case when $A$ is the Jacobian of $X$, and the map is the Abel--Jacobi morphism) works as follows. First, one shows that the image of the reduction map
\[
X(R)\cap A_{\tors }\to X(\overline{\F }_p )
\]
is finite, and that the reduction map restricted to $A_{\tors}(R)$ is injective. Next, one uses arithmetic jet spaces to bound the size of the image of $X(R)\cap A_{\tors }$ (the Buium step). Finally, one proves that (for $p$ sufficiently large), \textit{all} torsion points of $A$ lying on $X$ are unramified (the Coleman step). Hence one obtains a bound for $X(\overline{\Q }_p )\cap A_{\tors }$.

In \cite{DP25}, this argument was generalised to provide a new proof that Mordell implies Mordell--Lang for curves over number fields in their Jacobians, i.e. proving finiteness of $X(\overline{\Q } )\cap \Gamma $ for $\Gamma \subset A(\overline{\Q })$ the divisible hull of a finitely generated subgroup. The strategy is largely the same: one proves $X(\overline{\Q }_p )\cap \Gamma $ is unramified for all sufficiently large $p$, and bounds the size of the image of $X(R)\cap \Gamma $ in $X(\overline{\F }_p )$. However, in this case the reduction map
\[
X(R)\cap \Gamma \to X(\overline{\F }_p )
\]
need not be injective. Nevertheless, once one proves its image is finite, one can reduce to bounding the number of points of $X(R)\cap \Gamma $ lying on a fixed residue disc (i.e. with a fixed image in $X(\overline{\F }_p )$). This can then be translated into a question about $p$-adic abelian integrals. If the rank of $\Gamma $ is sufficiently small, then one can apply a `Chabauty--Coleman bound' to bound the size of $X(\overline{\Q }_p )\cap \Gamma $. In general, a naive Chabauty--Coleman method will not work because the topological closure of $\Gamma $ in $A(\overline{\Q }_p )$ will have infinite intersection with $X$. However, one can use a functional transcendence argument to prove that $X(\overline{\Q })\cap \Gamma $ consists of points of bounded degree, which together with the ramification bound above reduces Mordell--Lang to Mordell.

In \cite{approaches}, a similar argument was used to provide a new proof of Zilber--Pink for curves $X$ inside $\mathbb{G}_m ^3 $, which was first proved by Bombieri, Masser and Zannier \cite{BMZ99}. The strategy is similar, except that now in the `Coleman step' one can only prove that the ramification degree of points in $X\cap (\mathbb{G}_m ^3  )^{[2]}$ is uniformly bounded, and that the image of $(X(\overline{\Q }_p )-X(\Q _p ^{\nr }))\cap (\mathbb{G}_m ^3 )^{[2]}$ under the reduction map is finite. The fact that the underlying group is a torus introduces substantial simplifications, and the reader may find this a helpful introduction to the present paper.

\subsection{Plan of the paper}
The structure of the paper is as follows. In sections \ref{sec:1} to \ref{sec:finiteness}, we bound the size of $\red (X(R)\cap A^{[2]})$, proving part (1) of Theorem \ref{theorem:main}. As in \cite{buium:96}, this is proved using arithmetic jet spaces. One first translates the problem to bounding an intersection inside the arithmetic jet space $J^1 A_k $, which is a scheme over $\overline{\F }_p $ whose $\overline{\F }_p $ points are in bijection with the $R/p^2 R$ points of $A$. The power of arithmetic jet spaces for studying unlikely intersection problems comes from the fact that the image of $A^{[2]}(R)$ in $J^1 A_k$ is not Zariski dense. In sections \ref{sec:1} we describe its Zariski closure, and in sections \ref{sec:2} and \ref{sec:3} we describe its intersection with $J^1 X_k$. Buium's proof of finiteness of the Manin--Mumford intersection does not generalise to our setting, and in fact our proof of finiteness is inspired by Raynaud's proof of Manin--Mumford for curves in abelian varieties \cite{raynaud} (which also uses arithmetic jet spaces). Having proved finiteness of the intersection, we bound its size using intersection theory on a compactification of $J^1 A_k$ in a similar manner to \cite{buium:96}. In section \ref{sec:disk}, we bound the size of $X(R)\cap A^{[2]}$ on a fixed residue disc (assuming property (*)), using a Chabauty--Coleman-type bound for iterated integrals due to Betts \cite{betts:weight}. In section \ref{sec:ram} we prove the `Coleman step', bounding the ramified points. As in \cite{approaches}, we are unable to prove that $X(\overline{\Q }_p )\cap A^{[2]}$ is unramified, but instead we bound its ramification, and bound the image of its ramified points under the reduction map explicitly. Finally, in section \ref{sec:deg} we show how to bound the degree of points on $X(\overline{K})\cap A^{[2]}$ for $X\to A$ defined over a number field $K$.

\subsection{Notation}
Unless otherwise stated, $k:=\overline{\F }_p $, $R$ is the Witt vectors of $k$ (or equivalently the $p$-adic completion of the ring of integers of $\Q _p ^{\nr }$), and $K$ is its field of fractions. 

\subsection{Acknowledgements}
We would like to thank Arnab Saha for helpful conversations related to this paper. The first named author is funded by a Royal Society University Research Fellowship. The second named author is supported by Royal Society grant  RF$\backslash$ERE$\backslash$231161.
\section{Bounding unramified points via arithmetic jet spaces}\label{sec:1}
\subsection{Arithmetic jet spaces}
Given a scheme $X$ over a $p$-adic ring $R$, one can obtain a $p$-adic formal scheme by completion along the special fiber. Recall that the associated $p$-adic formal scheme, denoted by $\fX/R$, is a compatible system of schemes $X_{n}/\Spec(R/p^{n+1}R)$.

Buium introduced the theory of arithmetic jet spaces (also called $p$-jet spaces) in the category of $p$-adic formal schemes (cf.\ Section 1.4 in \cite{buium:96, buium:95}). Borger has independently developed the geometry of Witt vectors and a functor of points approach to arithmetic jet spaces in the category of schemes in \cite{bor1} \cite{bor2}, which is purely algebraic. Bertapelle, Previato, and Saha have reconciled these two approaches to arithmetic jet space theory by showing that Buium's ($p$-adic formal) jet spaces can be obtained as the completion of Borger's (algebraic) jet spaces along their special fibers, and they proved that their special fibers are isomorphic (cf.\ Corollary 3.13 and 3.14 in \cite{BPS}). In this paper, we only need to consider the special fiber of arithmetic jet spaces in our main applications; therefore, we may safely pass between the algebraic and $p$-adic formal categories whenever required. The usual derivation in jet space theory is replaced by a $p$-derivation in the arithmetic setting, which we define below:

\begin{definition}[$p$-derivation] 
Let $u: A\rightarrow B$ be a ring homomorphism. A $p$-derivation of $u$ is a set-theoretic map $\delta:A \rightarrow B$ that satisfies, for all $x,y \in A$:
\begin{enumerate}
    \item[\rm (i)] $\delta (1) = 0$,
    \item[\rm (ii)] $\delta (x+y) = \delta x + \delta y + C_p(u(x),u(y))$,
    \item[\rm (iii)] $\delta(xy) = u(x)^p \delta y + u(y)^p \delta x + p \delta x \delta y$,
\end{enumerate}
where
$$C_p(X,Y) =  \frac{X^p+Y^p -(X+Y)^p}{p}\in\mathbb{Z}[X,Y].$$ 
\end{definition}

Given a $p$-derivation, one can associate a unique lift of Frobenius (of $u$) $\phi: A\rightarrow B$ defined by 
$$\phi(x):=u(x)^{p}+p\delta(x).$$
Moreover, it is not difficult to see that there is a one-to-one correspondence between the $p$-derivations of $u$ and the ring homomorphisms $u_1:A\rightarrow W_{1}(B)$ such that $p_{1} \circ u_1=u$, given by $u_{1}(x)=(u(x),\delta(x))$, where $W_{1}(B)$ denotes the $p$-typical Witt vectors of length $2$ and $p_1: W_1(B)\rightarrow B$ is the first projection.

\medskip

The unique lift of Frobenius $\sigma :R\rightarrow R$ induces a $p$-derivation $\delta$ of the identity map defined by 
$$\delta(x):=\frac{\sigma(x)-x^p}{p}.$$


Let $Y$ be a smooth scheme over $R$. In \cite{buium:96}, Buium constructs the first arithmetic jet space of $Y\times  _R (R/p^2 R)$, which he denotes $Y^1 _0$, as a scheme over $k$, generalising the Greenberg transform. In this paper it will be convenient to follow Borger's approach \cite{bor1} \cite{bor2}. The first arithmetic jet space of $Y$ can be described via the functor of points as follows:

\begin{lemma}[Functor of points of Jet spaces] 
Let $Y/R$ be a scheme. The first arithmetic jet space of $Y/R$, denoted $J^1 Y$, represents the functor from the category of $R$-algebras to the category of sets given by 
$$J^1 Y (B)=Y(W_1(B)) \quad \text{for any } R\text{-algebra } B.$$
\end{lemma}

Hence, $J^1 Y$ is an $R$-scheme equipped with a projection map $u: J^1 Y \longrightarrow Y$. By Bertapelle--Previato--Saha (cf. Corollary 3.14 in \cite{BPS}), Buium's arithmetic jet space $Y^1_{0}$ is indeed the special fibre $J^1 Y_k$ of $J^1 Y$. Moreover, if $Y/R$ is a group scheme, then $J^1 Y_k$ is a $k$-group scheme. We will use this description throughout. \\

 Since $R$ is equipped with a $p$-derivation, it corresponds to an $R$-algebra map $R\rightarrow W_{1}(R)$. Thus, we have a functorial map $Y(R)\xrightarrow{\nabla^1} Y(W_1(R)).$ The functor of points of jet spaces then yields a map 
$$\nab: Y(R)\rightarrow J^1 Y(k)$$ 
via the compositions $\nab: Y(R)\xrightarrow{\nabla^1} Y^1(R)\xrightarrow{\mod p} J^1 Y(k)$ satisfying the following commutative diagram:
\begin{equation}\label{mod-p-com}
\vcenter{
\xymatrix{
Y(R)\ar[rd]_{\mod p} \ar[rr]^{\nab}& &  J^1 Y(k)\ar[ld]^{\overline{u}}\\
& Y(k)          }
}
\end{equation}

%

\subsection{Prolongation and universal property of the first jet space}\label{prolong}

Let $\fX$ and $\fY$ be $p$-adic formal schemes over $S=\Spf R$. We say a pair 
$(u,\delta)$ is a {\it prolongation}, and write 
$\fY \xrightarrow{(u,\delta)} \fX$, if $u: \fY \rightarrow \fX$ is a map of $p$-adic formal schemes 
over $S$ and $\delta: \sO_\fX \rightarrow u_*\sO_\fY$ is a 
$p$-derivation of $u$, which is a map of sheaf (of sets) making the following diagram commute: 
	$$
	\xymatrix{
	R \ar[r] &  u_* \sO_\fY \\
	R \ar[u]^\delta \ar[r] &  \sO_\fX \ar[u]_\delta \\
	} 
	$$ 
The prolongations forms a category with the evident morphisms, which we denote by $C^1.$	
By  \cite[Proposition 1.1]{buium:2000}, one can derive that the canonical prolongation $J^1\fX\xrightarrow{(u,\delta)} \fX$ satisfies the following 
universal property---for any prolongation $\fY_1 \xrightarrow{(u,\delta)} \fY_0$ and $\fX$ a $p$-adic formal
scheme over 
$R$, we have
\begin{equation}	
\label{canprouniv}
	\Hom(\fY_0,\fX) = \Hom_{C^1}(\fY_1, J^1\fX)
\end{equation}

\section{The Zariski closure of $A^{[2]}(R)$}\label{sec:1.5}
In this section, we generalise Buium's description of the Zariski closure of the image of $A(R)_{\tors }$ in $J^1 A_k $ to give a description of the Zariski closure of the Zariski closure of the image of $A^{[2]}(R)$ in $J^1 A_k $.
\begin{lemma}\label{lemma:dubious}
The number of $\overline{\F }_p $ subspaces of $\Lie (A_k )$ which arise as the image of the Lie algebra of an abelian subvariety of $A_K$ of codimension at least 2 is at most
\[
\sum _{e=1}^{g_A -2}\frac{p^{2\binom{g_A}{e}}-1}{p-1}.
\]
\end{lemma}
\begin{proof}
By duality, it is enough to prove that the number of $M_0 \subset H^1 _{\dR }(A_k )$ which arise as the kernel of a map
\[
H^1 _{\dR}(A_k )\to H^1 _{\dR}(B_k )
\]
pulling back along an abelian subvariety $B$ of codimension at least $2$ is at most
\[
\sum _{e=1}^{g_A -2}\frac{p^{2\binom{g_A }{e}}-1}{p-1}.
\]
$M_0$ is the reduction modulo $p$ of a $W(\overline{\F }_p )$ lattice $M$ in $H^1 _{\cris }(A_k /W(\overline{\F }_p ))$ of dimension $2e$.

$M$ defines a point in $\mathbb{P}(\wedge ^{2e} H^1 _{\cris }(A_k /W(\overline{\F }_p )))$ which is fixed by the action of Frobenius. Equivalently, $M$ defines an $F$-eigenspace of $\wedge ^{2e} H^1 _{\cris }(A_k /W(\overline{\F }_p ))$ which, by Poincar\'e duality, may be identified with the line spanned by its cycle class. Furthermore the reduction of this class modulo $p$ must lie in $F^e H^1 _{\dR}(A_k /k)$. The number of such $M_0$ is at most $\frac{p^{2\binom{g_A }{e}}-1}{p-1}$. Summing over $e$ gives the desired bound.
\end{proof}

\begin{proposition}\label{prop:kummer}
Let $A/R$ be as above. Let $W$ be a quotient of $A[p]$. Then the number of elements of $H^1 (K,W)$ which arise as the image of a $p$-torsion point of $A/B$, where $B<A$ is an abelian subvariety such that the $W=(A/B)[p]$ as quotients of $A[p]$, is at most $p^{g_A}$.
\end{proposition}
The proof of Proposition \ref{prop:kummer} will be given below. First we record some consequences.
\begin{proposition}\label{prop:finite_2}
For $A$ and $W$ as above, the number of elements of $A(R/p^2 R)/pA(R/p^2 R)+B(R/p^2 R)$ which map to torsion points in $A/B$, for $B<A$ with $(A/B)[p]\simeq W$, is at most $p^{g_A}$.
\end{proposition}
\begin{proof}
Suppose $B_1 $ and $B_2 $ are abelian subvarieties with $(A/B_i )[p]\simeq W$, that $P_i \in A(R)$ is such that $P_i$ maps to a torsion point in $A/B_i$, and that the classes of $P_i$ in $H^1 (K,W)$ are the same.

Let $f:A'\to A$ be the finite \'etale cover corresponding to $W$. Then we have an injection
\[
A(R)/f(A'(R))+pA(R)\hookrightarrow H^1 (K,W).
\]
On the other hand the map
\[
A(R)/pA(R)+B_1 (R)+B_2 (R) \to H^1 (K,W)
\]
is the composite of the inclusion
\[
A(R)/pA(R)+B_1 (R)+B_2 (R) \hookrightarrow A(R)/A'(R)
\]
induced by the map
\[
B_1 \times B_2 \to A'
\]
(such a map exists because under the map $B_1 \times B_2 \to A$, the image of $(B_1 \times B_2 )[p]$ is $\Ker (A[p]\to W)$) with the injection
\[
A(R)/A'(R)\to H^1 (K,W)
\]
coming from the exact sequence
\[
0\to M\to A' (\overline{K})\to A(\overline{K})\to 0.
\]
We deduce that the images of $P_1$ and $P_2 $ in $A(R)/pA(R)+B_1 (R)+B_2 (R)$ are the same. Hence the images in $A(R/p^2 R)/pA(R/p^2 R)+B_1 (R/p^2 R)+B_2 (R/p^2 R)$ are the same. The result follows from the isomorphisms
\[
pA(R/p^2 R)+B_1 (R/p^2 R)+B_2 (R/p^2 R)\simeq pA(R/p^2 R)+B_i (R/p^2 R)
\]
for $i=1,2$. Indeed, since $B_1 [p]=B_2 [p]$ as subspaces of $A[p]$, we deduce $\Lie (B_{1,k})=\Lie (B_{2,k})$ as subspaces of $\Lie (A_k )$, since the morphisms of group schemes
\[
B_i [p]_k \hookrightarrow B_{i,k}
\]
induce isomorphism on Lie algebras. Hence the images of $J^1 B_{1,k}$ and $J^1 B_{2,k}$ in $J^1 A_k (k)/pJ^1 A_k (k)$ are the same.
\end{proof}
\begin{proposition}
For $A$ and $W$ as above, the number of elements of $A(R/p^2 R)/pA(R/p^2 R)+B(R/p^2 R)$ coming from $R$-points of $A$ mapping to torsion points of $A/B$ is at most $p^{g_A}$.
\end{proposition}
\subsection{Proof of Proposition \ref{prop:kummer}}
Recall that the class of $P$ in $H^1 (K,A[n])$ is the image of $P$ under the connecting homomorphism associated to
\[
0\to A[n]\to A(\overline{K})\stackrel{\cdot n}{\longrightarrow }A(\overline{K})\to 0.
\]
It follows that, in the case of an $m$-torsion point $P\in A[m](K)$, we may take the image under the connecting homomorphism associated to 
\[
0\to A[n]\to A[nm]\to A[m]\to 0.
\]
Taking limits, we obtain a class $\kappa (P)$ in $H^1 (K,T_p A)$, which in the case of a $p^n$-torsion point comes from the connecting homomorphism associated to
\[
0\to T_p A \to p^{-n}T_p A\to A[p^n ]\to 0.
\]
If $A$ is an abelian variety over $R$, then we have a connected-\'etale exact sequence
\[
0\to T\to T_p A \to \Z _p ^{r} \to 0,
\]
where $r$ is the $p$-rank of $A$, and $H^0 (K,T\otimes \Q _p /\Z _p )=0$. It follows that the class of any $p^n$-torsion point in $H^1 (K,T_p A)$ is in fact in the image of $H^1 (K,T)$. The connected-\'etale exact sequence gives class $\theta _1 ,\ldots ,\theta _r  \in H^1 (K,T)$, from pulling back the exact sequence along the basis vectors $e_1 ,\ldots ,e_r \in \Z _p ^r$. To complete the proof of Proposition \ref{prop:kummer} it will be enough to prove the following Lemma, since the $\Z _p$-rank of the \'etale part of the Tate module of $A$ is at most $g_A$.
\begin{lemma}
The class $\kappa (P)$ comes from a class in $H^1 (K,T)$ which is in the $\Q _p $-span of $\theta _1 ,\ldots ,\theta _r $.
\end{lemma}
\begin{proof}
Specifically, suppose that the image of $P \in A[p^n ]$ in the \'etale quotient of $A[p^n ]$ is the reduction modulo $p^n$ of $\sum \lambda _i e_i $. Then $\kappa (P)$ lifts to a class $\xi \in H^1 (K,T)$ such that 
\[
p^n \xi =\sum \lambda _i \theta _i .
\]
This follows from unravelling the definition of $\kappa (P)$. Indeed, by definition $\kappa (P)$ is the image of $P\in H^0 (K,A[p^n ])$ under the boundary map associated to
\[
0\to T_p A \to p^{-n}T_p A\to A[p^n]\to 0.
\]
Let $U\subset p^{-n}T_p A$ be the pre-image of the subgroup $\overline{U}$ of $H^0 (K,A[p^n ])$ spanned by $P$. Then $\kappa (P)$ is equivalently the image of the generator of $\overline{U}$ under the boundary map associated to 
\[
0\to T_p A\to U\to \overline{U}\to 0.
\]
Then $U$ also has a connected-\'etale exact sequence
\begin{equation}\label{eqn:more}
0\to T\to U \to U'\to 0,
\end{equation}
where $\Z _p ^r \subset U' \subset p^{-n}\Z _p ^r $. The generator of $U'$ lifts to an element $p^{-n}\sum \lambda _i e_i $ in $U'$. From the commutative diagram
\[
\begin{tikzcd}
0 \arrow[r] & T \arrow[d] \arrow[r] & U \arrow[d] \arrow[r] & U' \arrow[d] \arrow[r] & 0 \\
0 \arrow[r] & T_p A \arrow[r]           & U \arrow[r]           & \overline{U} \arrow[r]           & 0
\end{tikzcd}
\]
we see that $\kappa (P)$ lifts to an element $\xi $ in $H^1 (K,T)$ which is the image of $p^{-n}\sum \lambda _i e_i \in H^0 (K,T)$ under the boundary map associated to \eqref{eqn:more}. On the other hand, from the commutative diagram
\[
\begin{tikzcd}
0 \arrow[r] & T \arrow[d] \arrow[r] & T_p A \arrow[d] \arrow[r] & \Z _p ^r \arrow[d] \arrow[r] & 0 \\
0 \arrow[r] & T \arrow[r]           & U \arrow[r]           & U' \arrow[r]           & 0
\end{tikzcd}
\]
we see that $p^n \xi $ is the image of $\sum \lambda _i e_i$ under the boundary map associated to the connected -\'etale sequence, which by definition is $\sum \lambda _i \theta _i $.
\end{proof}
In the language of arithmetic jet spaces, we have established the following.

\begin{proposition}\label{prop:zar}
The Zariski closure of the image of $A^{[2]}(R)$ in the $g$-dimensional vector group $J^1 A_k /\mathcal{M}$ is contained in a subset of the form $\cup _{i=1}^n \cup _{j=1}^{m_i} (V_i +\gamma _{ij} )$, where $V_i$ is a vector subspace of codimension at least 2, $\gamma _j$ is a point of $J^1 A_k /\mathcal{M}$, and 
\begin{align*}
n & \leq \sum _{e=1}^{g_A -2}\frac{p^{2\binom{g_A }{e}}-1}{p-1}\\
m_i & \leq p^{g_A} .
\end{align*}
\end{proposition}
\begin{proof}
Via the identification $(J^1 A_k /\mathcal{M})(k)\simeq A(R/p^2 R)/p\cdot A(R/p^2 R)$, this follows from Lemma \ref{lemma:dubious} and Proposition \ref{prop:finite_2}.
\end{proof}

\section{Finiteness of unramified points: non-supersingular case}\label{sec:2}
In this section we prove that the image of $X(R)\cap A^{[2],o}(R)$ in $X_k (R)$ is finite. Here $A^{[2],o}(R)$ is defined to be the union of all $R$-points of subgroups $B\subset A$ such that the special fibre of $A/B$ is not supersingular. First, we note an equivalent way to characterise the condition that $A/B$ is not supersingular.
\begin{lemma}\label{lemma:equivalent}
For an abelian subvariety $B<A$, the following are equivalent.
\begin{enumerate}
\item $A/B$ has supersingular reduction.
\item The map $F:\Lie (A_k /B_k )\to \Lie (A_k /B_k )^{(p)}$ is zero.
\item The image of $F:\Lie (A_k )\to \Lie (A_k )^{(p)}$ is contained in the image of $\Lie (B_k )$.
\end{enumerate}
\end{lemma}
\begin{proof}
The equivalence of (1) and (2) is well known (see e.g. \cite{coleman:ramified}). The equivalence of (2) and (3) follows from the compatibility of the action of $F$ with the short exact sequence
\[
0\to \Lie (B_k )\to \Lie (A_k )\to \Lie (A_k /B_k )\to 0.
\]
\end{proof}

\subsection{The universal vectorial extension and the first jet space}
In this subsection, we work in the $p$-adic formal category. We first recall and slightly expand upon the relation between $J^1 A_k$ and the universal vector extension of $A_k$ discussed in \cite[\S 5]{pandit:26}. Let $\mathfrak{A}$ be the formal completion of $A$ along its special fibre.  The universal vectorial extension $E(\mathfrak{A})$ of $\mathfrak{A}$ sits in a short exact sequence of formal group schemes:
\begin{align} \label{u-ext} 
0\longrightarrow V(\mathfrak{A}) \longrightarrow E(\mathfrak{A}) \xrightarrow{u_e} \mathfrak{A} \longrightarrow 0
\end{align}
where $V(\mathfrak{A})$ is the (formal) vector group associated to the $R$-module $H^{0}(A^{\vee}, \Omega_{A^{\vee}})$. The invariant differentials of the universal vectorial extension are canonically isomorphic to the first crystalline cohomology of $A_k$ over $R$, and dually the Lie algebra of $E(\mathfrak{A})$ is canonically isomorphic to $H^1 _{\cris }(A_k ^\vee /R)$. By \cite{MM}, the universal extension satisfies the following universal properties: 
\begin{enumerate}
\item Given any extension $E$ of $\mathfrak{A}$ by a vector group $M$, there exists a unique morphism $f: V(\mathfrak{A}) \longrightarrow M$ such that $E$ can be obtained as the pushforward of \eqref{u-ext} by $f$.
\item Given abelian schemes $A_1 ,A_2 $ over $R$ with completions $\mathfrak{A}_1 ,\mathfrak{A}_2 $ along their special fibres, and a morphism
\[
f: A_{1,k}\to A_{2,k}
\]
of abelian varieties over $k$, there is a unique morphism
\[
\EE(f): E(\mathfrak{A}_{1})\to E(\mathfrak{A}_2)
\]
of their universal extensions, such that the pullback map on invariant differentials
\[
\Omega _{E(\mathfrak{A}_2 )|R}\to \Omega _{E(\mathfrak{A}_1 )|R}
\]
is isomorphic to the pullback map on crystalline cohomology
\[
H^1 _{\cris }(A_{2,k}/R)\to H^1 _{\cris }(A_{1,k} /R) 
\]
via the isomorphisms above.
\end{enumerate}

\subsubsection{Frobenius compatibility} Let $C$ be an $R$-algebra. Then $C^{\sigma }$ denotes the $R$-algebra $C\times _{R,\sigma }R$. Given a $p$-adic formal scheme $\fX$ over $R$ by functor of points we can define $\fX^\sigma$ as follows: for any $R$-algebra $C$
\[\fX^\sigma(C)=\fX(C^\sigma).
\]
which is indeed the pull-back of $\fX$ along $\sigma$ as below
\[
\xymatrix{
\fX^{\sigma}\ar[d]\ar[r]^-{\pi_\fX}& \fX\ar[d]\\
R\ar[r]^-{\sigma} & R
}
\] 
Note that $\fX^\sigma_k\simeq X_k^{(p)}$ as $k$-schemes.
By the functor of points interpretation, there is a functorial map
\[
J^1 (\fX)\to \fX\times \fX^{\sigma }
\]
corresponding to the ghost map
\begin{align*}
W_1 (C)&\to C\times C^{\sigma }\\
(c_0, c_1)& \to (c_0, c_0^p+p c_1)
\end{align*}
From the second universal property of the universal vectorial extension,  we obtain a unique morphism 
\[
\EE(F):E(\mathfrak{A})\to E(\mathfrak{A}^{\sigma })
\]
lifting the Frobenius $F:A_k \to A_k ^{(p)}=(A^{\sigma })_k $. In particular, we obtain a morphism by the following composition
\[
\theta: E(\mathfrak{A})\xrightarrow{\EE(F)} E(\mathfrak{A}^{\sigma}) \xrightarrow{u'_e} \mathfrak{A}^{\sigma}\xrightarrow{\pi_{\fA}}\fA.
\]
Note that by functoriality we have $E(\fA)_k\simeq E(A_k),$ giving the following commutative diagram
\[
\xymatrix{
E(A_k)\ar[d]_{u_{e,k}}\ar[r]^-{\EE(F)_k}& E(A_k^{(p)})\ar[d]_{u'_{e,k}} \\
A_k \ar[r]^{F} &A_k^{(p)}
}
\]
Let $\theta_k$ denote the map $\pi_{A_k}\circ u'_{e,k}\circ \EE(F)_k=\pi_{A_k}\circ F\circ u_{e,k}.$ Since the map $\pi_{A_k}\circ F:A_k\to A_k$ is the absolute Frobenius, which is identity on the topological space, we have  $\theta_k=u_{e,k}: |E(A_k)|\to |A_k|$ as the maps of underlying topological spaces. Moreover, the underlying topological space of a $p$-adic formal scheme is identical to its special fiber, we have 
$$\theta=u_e:|E(\fA)|\to |\fA|.$$
Therefore $\theta$ will induce a morphism of sheaf (of sets) $\theta^*:\sO_{\fA}\to {u_e}_*\sO_{E(\fA)}$ satisfying the following: for any open set $U\subset \fA,$ we have $\theta^*_U: \sO_\fA(U) \to \sO_{E(\fA)}(u_e^{-1}(U))$ such that $$\theta^*_U(h)={u_e^*(h)}^p \mod(p),$$ showing that $\theta^*$ is a lift of Frobenius with respect $u_e^*.$  Since $\fA$ and  $E(\fA)$ are smooth group schemes, we have the structure sheaves to be $p$-torsion free and we can define a $p$-derivation $\delta_{E(\fA)}$, which is a map of sheaf of sets as follows: for any open set $U,$ we define
$$\delta_U(h)=\frac{\theta^*_U(h)-{u^*_e(h)}^p}{p}.$$
This in particular proves that $(u_e, \delta_{E(\fA)}): E(\fA) \to \fA$ is a prolongation as introduced in Section \ref{prolong}.
Hence by the universal property of the first jet space as in \eqref{canprouniv}, we have the unique map of prolongations 
\[
\eta :E(\mathfrak{A})\to J^1 \mathfrak{A}.
\]
Also note that $\fN^1:=\Ker[J^1\fA\xrightarrow{u} \fA]$ is a (formal) vector group over $R$ (cf. Proposition $2.2$ and Lemma $2.3$ in \cite{buium:95}). Hence any two morphisms $f,g: E(\fA) \to J^1\fA$ lying over $A$ will coincide (as their difference will factor through a vector group which implies the map is trivial, cf. Lemma $5.1$ in \cite{pandit:26}). Therefore $\eta$ is also the unique morphism of vector extensions of $\mathfrak{A}$ induced by the first universal property of $E(\fA).$

\begin{proposition} We have the following commutative diagram of $p$-adic formal schemes
\[
\xymatrix{
E(\mathfrak{A})\ar[d]_{\eta}\ar[rr]^-{(u_e, u'_e \circ\EE(F))}&&\mathfrak{A}\times \mathfrak{A}^\sigma\\
J^1\mathfrak{A}\ar[rru] _{w_1}
}
\]

\end{proposition}
\begin{proof} It is enough to prove that $u'_e \circ\EE(F)=\pi_2\circ w_1\circ \eta: E(\fA)\to \fA^\sigma.$ In fact, it is enough to show that the associated absolute morphisms are same, i.e, $\pi_{\fA}\circ u'_e \circ\EE(F)=\pi_{\fA}\circ\pi_2\circ w_1\circ \eta: E(\fA)\to \fA^\sigma\to \fA,$ because their pull-back along $\sigma$ will recover the desired identity by the universal property of the fiber product of schemes. 

Moreover, it is not difficult to see that for smooth schemes (which have $p$-torsion free structure sheaf) a $p$-derivation of $u_e$ of sheaf of (sets) will uniquely corresponds to a lift of Frobenius of sheaf (of rings) by the discussion above. The map $\eta: E(\fA)\to J^1\fA$ is a morphism of prolongations, which in particular implies $\delta_{J^1\fA}\circ \eta^*=\delta_{E(\fA)}: \sO_{\fA}\to {u_e}_*\sO_{E(\fA)}$ as $p$-derivations of structure sheaf (of sets).  Therefore $\eta$ commutes with the associated absolute Frobenius morphisms, and we are done.
\end{proof}

\begin{proposition}\label{prop:doesnt_exist_yet}
The commutative diagram
\[
\xymatrix{
E(\mathfrak{A})\ar[d]_{\eta}\ar[r]&\mathfrak{A}\times \mathfrak{A}^\sigma \\
J^1\mathfrak{A}\ar[ru] _{w_1}
}
\]
induces maps on invariant differentials
\[
\begin{tikzcd}
e^* \Omega _{E(\mathfrak{A}|R} \simeq H^1 _{\dR}(A/R) & H^0 (\mathfrak{A},\Omega )\oplus H^0 (\mathfrak{A}^{\sigma },\Omega ) \arrow[l, "{(1,F_{\cris })}"'] \arrow[ld, "{(1,p)}"] \\
e^* \Omega _{J^1 \mathfrak{A}|R}\simeq  H^0 (\mathfrak{A},\Omega ) \oplus H^0 (\mathfrak{A}^{\sigma },\Omega ) \arrow[u, "{(1,F_{\cris }/p)}"] &
\end{tikzcd}
\]
\end{proposition}
\begin{proof}
Indeed, that the ghost map $w_1$ at the level of differentials is given by $(\mathrm{id}, p)$ can be seen by choosing an {\'e}tale coordinate, as computed in Section~6 of \cite{BS}. The map on invariant differentials induced by $\theta $ is a consequence of the universal property. Consequently, the remaining map is uniquely determined. 
\end{proof}

\subsection{Application to transversality}
As in \cite{buium:96}, our study of unlikely intersections in $A$ makes essential use of the maximal abelian subvariety of $J^1 A_k$, which we will denote by $\mathcal{M}$. By the previous subsection, $\mathcal{M}$ can also be characterised as the image of the universal vector extension $E(A_k ^{(p)})$ under the composite map
\[
E(A_k ^{(p)})\to E(A_k )\to J^1 A_k
\]
where the first map is the unique lift $E(V)$ of the Verschiebung on $A_k$ which gives a commutative diagram
\[
\begin{tikzcd}
\Lie (E(A_k ^{(p)})) \arrow[d, "\simeq "] \arrow[r, "E(V)"] & \Lie (E(A_k )) \arrow[d, "\simeq "] \\
H^1 _{\dR}(A_k ^{\vee ,(p)} /k) \arrow[r, "F^*"]                & H^1 _{\dR}(A_k ^\vee /k).               
\end{tikzcd}
\]

We now fix a point $\gamma \in A(R)_{\tors }$ and a vector subgroup $V_0 \subset \Lie (A_k ^{(p)})$ of codimension at least 2. By Lemma \ref{lemma:equivalent}, to prove finiteness of the image of $X(R)\cap A^{[2],o}(R)$ in $X_k$ it will be enough to prove the following.
\begin{proposition}\label{prop:finite}
Suppose the image of $F  :\Lie (A_k )\to \Lie (A_k ^{(p)})$ is not contained in $V_0 $. Then the projection of $J^1 X_k $ to $V_0 +\mathcal{M}+\gamma $ is finite.
\end{proposition}
It will be enough to bound the number of points $\overline{x}\in X_k (k)$ such that there exists $x \in J^1 X_k  (k)$ above $\overline{x}$ such that $x\in V_0 +\mathcal{M}+\gamma  $ and $T_x X_k ^1 \cap T_x (V_0 +\mathcal{M}+\gamma )\neq 0$. To do this we use the previous subsection to give a description of the tangent space of the maximal abelian subvariety of $J^1 A_k $. In the course of doing this, we will give a simple description of the tangent space of the first arithmetic jet space of an arbitrary smooth scheme over $R$. This will be a consequence of the following elementary calculation.

\begin{lemma}\label{lemma:witt}
Let $S$ be an $\F_p $-algebra. Then we have a functorial isomorphism
\[
W_1 (S[\epsilon ]/\epsilon ^2 ) \simeq W_1 (S)[\epsilon _1 ]/(\epsilon _1 ^2 ,p\epsilon _1 )\otimes _{W_1 (S),F }W_1 (S[\epsilon _2 ]/(\epsilon _2 ^2 ,p\epsilon _2 )
\]
where the tensor product is via the natural inclusion 
\[
W_1 (S)\hookrightarrow W_1 (S)[\epsilon _1 ]
\]
and the composite
\[
W_1 (S)\stackrel{F}{\longrightarrow }W_1 (S)\hookrightarrow W_1 (S)[\epsilon _2 ]
\]
of Frobenius on $W_1 (S)$ with the natural inclusion.
\end{lemma}
\begin{proof}
This can be checked directly at the level of Witt coordinates: the desired indentification is given by sending $\epsilon _1 $ to $(\epsilon ,0)$, and $\epsilon _2 $ to $(0,\epsilon )$.
\end{proof}
\begin{lemma}\label{lemma:nice}
For any smooth $X/R$, and $x\in X(R)$ mapping to $\overline{x}\in X(k)$, there is a functorial isomorphism
\[
T_{\nabla ^1 _0 (x)}X_k ^1 \simeq T_{\overline{x}}X_k \oplus T_{\phi (\overline{x})}X_k ^{(p)}.
\]
\end{lemma}
\begin{proof}
This follows from the functor of points interpretation of $X_k ^1$. Indeed we have
\begin{align*}
TX_k ^1 (k)\simeq X_k ^1 (k[\epsilon ]/\epsilon ^2 ) \\
\simeq X(W_1 (k[\epsilon ]/\epsilon ^2 )).
\end{align*}
Hence Lemma \ref{lemma:witt} gives the desired identification.
\end{proof}

\begin{lemma}\label{lemma:imageM}
The map
\[
E(A^{(p)}_k )\to J^1 A_k 
\]
is given, on tangent spaces, by the map
\[
(F\circ \pi _A ,\pi _A ):\Lie (E(A^{(p)}_k ))\to \Lie (J^1 A_k )\simeq \Lie (A_k )\oplus \Lie (A_k ^{(p)}),
\]
where $\pi _A$ denotes the projection
\[
\Lie (E(A^{(p)}_k ))\to \Lie (A_k ).
\]
In particular, the Lie algebra of $\mathcal{M}\subset J^1 A_k $ is equal to the graph of the Frobenius morphism $\Gamma _F \subset \Lie (A_k )\oplus \Lie (A_k ^{(p)})$.
\end{lemma}
\begin{proof}
This follows from Lemma \ref{lemma:nice} and Proposition \ref{prop:doesnt_exist_yet}.
\end{proof}

Let $\sigma _0$ denote the induced map
\[
\sigma _0 :X_k \to \mathbb{P}(\Lie (A_k ) )
\]
and let $\sigma _1$ denote the induced map
\[
\sigma _1 :X_k ^{(p)} \to \mathbb{P}(\Lie (A_k ^{(p)}) )
\]
obtained by conjugating by Frobenius.
Let $W_0$ be a vector subgroup of $\Lie (A_k )\times \Lie (A_k ^{(p)})$ of dimension at most $g-2$ corresponding to $V_0$. 
\begin{lemma}
Under the identification
\[
\Lie (J^1 A_k )\simeq \Lie (A_k )\oplus \Lie (A_k ^{(p)}), 
\]
we have
\[
T_{\nabla ^1 _0 (x)} J^1 X_k \simeq (T_{\overline{x}}X_k \oplus T_{\phi (\overline{x})}X_k ^{(p)} )
\]
and
\[
T_{\nabla ^1 _0 (x)} (V_0 +\mathcal{M}+\gamma )\simeq W_0 +\Gamma _F .
\]
\end{lemma}
\begin{proof}
The first isomorphism is a special case of Lemma \ref{lemma:nice}. The second isomorphism follows from the isomorphism
\[
\Lie (\mathcal{M})\simeq \Gamma _F ,
\]
which is a consequence of Lemma \ref{lemma:imageM}.
\end{proof}
If $W_0 +\Gamma _F $ contains $\Lie (A_k )$, then the intersection will always be positive dimensional. 
To conclude the proof of Proposition \ref{prop:finite} it is enough to show that, if $W_0 +\Gamma _F $ does not contain $\Lie (A_k )$, then there will only be finitely many points of intersection . The proof will follow a similar approach to \cite[Corollary 1.3.4]{raynaud}.
\begin{lemma}
Assume $\Lie (A_k )$ is not a subspace of $W_0 +\Gamma _F $. Then there are only finitely many points $x\in X_k (k)$ such that $(\sigma _0 (x) ,\sigma _1 (\phi (x) ))$ lies in the image of $W_0 +\Gamma _F$ in $\mathbb{P}(\Lie (A_k )) \times \mathbb{P} (\Lie (A_k ^{(p)}))$.
\end{lemma}
\begin{proof}
Let $W_1$ denote the quotient of $\Lie (A_k )\times \Lie (A_k ^{(p)})$ by $W_0 +\Gamma _F$. 
Note that $W_1$ is isomorphic to the quotient of $\Lie (A_k ^{(p)})$ by the image of $W_1$ under the map
\[
(F ,1):\Lie (A_k )\times \Lie (A_k ^{(p)})\to \Lie (A_k ^{(p)}).
\]
Viewing $W_1 $ as a subspace of $\Lie (A_k ^{(p)})$ in this way, we obtain maps
\begin{align*}
& \overline{\sigma }_0 :X_k \stackrel{\sigma _0 }{\longrightarrow }\mathbb{P}(\Lie (A_k )) \stackrel{F }{\longrightarrow }\mathbb{P}(\Lie (A_k ^{(p)})) \dashrightarrow \mathbb{P}(\Lie (A_k ^{(p)})/W_1 ), \\
& \overline{\sigma }_1 :X_k \stackrel{\phi }{\longrightarrow }X_k ^{(p)}\stackrel{\sigma _1 }{\longrightarrow }\mathbb{P}( \Lie (A_k ^{(p)})) \dashrightarrow \mathbb{P}(\Lie (A_k ^{(p)})/W_1 ).
\end{align*}
To prove the Lemma, it is enough to show that $\overline{\sigma }_0 \neq \overline{\sigma }_1 $. This follows from the fact that the degree of $\overline{\sigma }_0 $ is $2g_X-2$, and the degree of $\overline{\sigma }_1 $ is $p(2g_X-2)$.
\end{proof}

\section{Finiteness of unramified points: supersingular case}\label{sec:3}
In this section we prove the following result.
\begin{proposition}\label{prop:finites}
Let $A^{[2],s}(R)$ denote the union of all $R$-points of subgroups $B\subset A$ such that the special fibre of $A/B$ is supersingular. Then $\red (A^{[2],s}(R)\cap X(R))$ is finite.
\end{proposition}
The proof of this result is rather different from that of the previous section, and introduces methods which may be of independent interest. The divergence with the non-supersingular case arises from the fact that the Zariski closure of $A^{[2],s}(R)$ in $A_k ^1$ with $X_k ^1 $ is \textit{nowhere} transverse. To prove finiteness, we need delve further into the formal completion of a point in the intersection (specifically, to depth $p$).

We let $\mathbb{X}_n (\mathfrak{A}/R)$ denote the $R$-module of $\delta$-characters, i.e. homomorphism of formal schemes
\[
J^n  (\mathfrak{A}/R)\to \widehat{\mathbb{G}}_{a,R}.
\]
where $\widehat{\mathbb{G}}_{a,R}$ denotes the formal completion of $\mathbb{G}_{a,R}$ along its special fibre. 
Similarly let $\mathbb{X}_n (A_{\overline{\F }_p }/\overline{\F }_p )$ denote the $\overline{\F }_p $-vector space of homomorphisms of schemes
\[
J^n  (A_{\overline{\F }_p }/\overline{\F  }_p )\to \mathbb{G}_{a,\overline{\F }_p }.
\]
We have obvious maps
\[
\mathbb{X}_n (\mathfrak{A}/R)\to \mathbb{X}_n (A_{\overline{\F }_p }/\overline{\F }_p )
\]
and
\[
\mathbb{X}_n (\mathfrak{A}/R)\to \mathbb{X}_n (A_K ^{\an } /K ),
\]
where $\mathbb{X}_n (A_K ^{\an } /K)$ denotes the $K$-vector space of homomorphisms
\[
J^n A_K ^{\an} \to \mathbb{G}_{a,K}^{\an }
\]
from the Berthelot generic fibre of the formal scheme $J_n \mathfrak{A}$ to the analytification of $\mathbb{G}_{a,K}$. We use the fact, proved in \cite{bor2}, essentially in \cite{buium:95} and also in a slightly different language in \cite{DS26}, that $J_n A_K ^{\an }$ is naturally a rigid analytic subspace of the analytification of $\prod _{i=0}^n A^{\sigma ^i }$. The following construction of elements of $\mathbb{X}_n (A_K ^{\an } /K)$ is proved in \cite{DS26}.
\begin{lemma}[\cite{DS26}, Corollary 4]
If, for $0\leq i\leq n$, $\omega _o \in H^0 (A^{\sigma ^i },\Omega )$ have the property that $\sum (\phi ^* )^i \omega _i =0$ in $H^1 _{\dR}(A/K)$, then the restriction of $\int _0 \sum \pi _i ^* \omega _i $ to $J^n X^{\an }$ defines an element of $\mathbb{X}_n (A/K)$.
\end{lemma}
We deduce the following formula for $\delta $-characters of order 2 on supersingular abelian varieties.
\begin{lemma}
If $A$ is supersingular, then for every $\omega \in H^0 (A,\Omega )$ there are unique differentials $\eta =\eta (\omega )$ in $H^0 (A^{\sigma },\Omega )$ and $\xi = \xi (\omega )$ in $H^0 (A^{\sigma ^2 },\Omega )$, with $\val (\xi )=\val (\eta )=\val (\omega )$, such that 
\[
(\phi ^* )^2 (\xi ) +p\phi ^* (\eta ) +p\omega =0.
\]
\end{lemma}
\begin{proof}
If $A$ is supersingular, then 
\[
H^1 _{\dR}(A/R)=\frac{V}{p}H^0 (A,\Omega )\oplus H^0 (A,\Omega )
\]
and 
\[
V(H^0 (A,\Omega ))\subset pH^1 _{\dR}(A,\Omega )
\]
(see e.g. \cite{coleman:ramified}).
From the first equation, we deduce the existence, for any $\xi _0 $, of a unique $\eta  _0$ and $\omega _0$ in $H^0 (A,\Omega )$ such that
\[
\phi ^2 \xi _0 +\phi \eta _0 +\omega _0 =0.
\]
Since $\phi (H^0 (A^{\sigma },\Omega ))\subset pH^1 _{\dR}(A/R)$, we in fact have $\omega _0 \in pH^0 (A,\Omega )$ so there is an $\omega _1 \in H^0 (A,\Omega )$ such that
\[
\phi ^2 \xi _1 +\phi \eta _1 +p\omega _1 =0.
\]
This relation gives
\[
\phi (\xi _1 )+\eta _1 +V(\omega _1 )=0.
\]
Hence we must actually have $\eta _1 \in pH^0 (A^{\sigma },\Omega )$.
\end{proof}
Let $F_{\omega }$ denote the function
\[
\frac{1}{p}(\int _0 \pi _2 ^* \xi +p\pi _1 ^* \eta +p\omega ).
\]
\begin{lemma}
If $\omega \in H^0 (A,\Omega _{A|R})$ is nonzero modulo $p$, then $F_\omega $ comes from an element of $\mathbb{X}_2 (\mathfrak{A}/R)$, and its reduction modulo $p$ is a nonzero element of $X_1 (A_k /k)$. 
\end{lemma}
\begin{proof}
To check integrality it is enough to work locally. We use Buium's local coordinates for $J^n (\mathfrak{A})$ \cite{buium:95}.  Namely, let $x_i $ be integral parameters at a point $z_0$. Let $\phi $ be a lift of Frobenius near $z_0 $ given by $x_i \mapsto x_i ^p$, and let $\sigma \in \Gal (K|\Q _p )$ be the lift of the absolute Frobenius on $\overline{\F }_p $. Then, near $(z_0 ,\sigma (z_0 ),\sigma ^2 (z_0 ))$, $J^2 (A)^{\an }$ can be described as the rigid analytic subspace of $A\times A^\sigma \times A^{\sigma ^2 }$ where 
\[
|\sigma (x_i ) -x_i ^p |\leq \frac{1}{p}
\]
and 
\[
|\sigma ^2 (x_i )-\sigma (x_i )^p-p(\frac{\sigma (x_i )-x^p }{p})^p|\leq \frac{1}{p^2 }.
\]
It follows that, to check that $F_\omega $ comes from $X_2 (A/R)$, it is enough to check that for all $z_0$ as above, its expansion lies in $R[\! [x_i ,x_i '  ,x_i '' ]\! ]$, where
\[
x_i ' :=\frac{\sigma (x_i ) -x_i ^p}{p}
\]
and 
\[
x_i '' := \frac{\sigma ^2 (x_i )-(x_i ^p +px_i ')^p +p(x_i ')^p }{p^2 }.
\]
We may rewrite $F_{\omega }$ in terms of these coordinates. Let 
\begin{align*}
\omega =\sum a_i x^i  dx_j  \\
\eta = \sum b_i x^i  dx_j  \\
\xi = \sum c_i x^i dx_j .
\end{align*}
where each sum is over $1\leq j\leq g$ and $i$ in $\Z _{\geq 0}^g$. We have
\[
d\sigma ^2 (x_i )=-p^2 x_i ^{p^2 -1}dx_i +p^2 dx_i '' +2d(x_i ^p +px_i ')^p
\]
Then 
\[
\pi _1 ^* \eta = \sum _{j,n} \sigma (b_n )\prod (px_i ' +\phi (x_i ))^{n_i } d\sigma (x_j )
\]
\[
\pi _2 ^* \eta = \sum _{j,n} \sigma ^2 (c_n ) \prod (p^2 x_i ''+(x_i ^p +px_i ')^p +p(x_i ')^p )^{n_i }d\sigma ^2 (x_j ).
\]
\end{proof}
As a consequence of the proof, we deduce the following formula for the restriction of $F$ to $X$ mod $p$.
\begin{lemma}\label{lemma:F}
Suppose 
\[
(\phi ^* )^2 \xi +p\phi ^* (\eta )+p\omega =pdG.
\]
We deduce that, mod $p$,
\begin{align*}
& \frac{1}{p}(\int _0 \pi _2 ^* \xi +\pi _1 ^* \eta +p\omega )  \\
& \equiv G(x)+g(x)^{p^2 } (x')^p ,
\end{align*}
where $\xi =g(x)dx$. 
\end{lemma}
Let $t:=(x')^p$. Then we see that $F_{\omega }$ is Zariski locally (say, on $U\subset J^1 X_k $) a function of $x$ and $t$. Let $W\subset H^0 (A_k ,\Omega )$ be a subspace coming from the kernel of a map
\[
H^0 (A_k ,\Omega )\to H^0 (B_0 ,\Omega ),
\]
where $B_0 \subset A_k$ is the reduction modulo $p$ of an abelian subvariety of $A$ such that $A_k /B_0 $ is supersingular. To prove finiteness of common zeroes of $F_{\omega }$ (for $\omega $ in $W$) it is enough to prove transversality of the intersection as a function of $x$ and $t$.  By Lemma \ref{lemma:F}, we have
\[
dF_{\omega }\equiv \omega +g(x)^{p^2 }dt.
\]
Note that $\xi \equiv \frac{V^2 }{p}\omega $ modulo $p$.
Let $Z$ be the Zariski closure of the union of the image of all such $J^1 B_0$ in $J^1 A_k $. From the description of the $\delta $-characters above, we find that points $z\in U$ such that the intersection of $J^1 X_k $ with $Z$ is not transverse map to points $\overline{z}$ in $X_k$ where the two maps
\[
\sigma _0 ,\sigma _1 :X_k \to \mathbb{P}(W^* )
\]
coincide, where the first map is the canonical one, and the second map is the composite of the canonical map with
\[
\mathbb{P}(W^* )\stackrel{\phi ^2 }{\longrightarrow } \mathbb{P}(W^* )^{(p^2 )}\stackrel{V^2 /p}{\longrightarrow }\mathbb{P}(W^* ),
\]
where the first map is the relative Frobenius squared, and the second is the degree 1 map induced by the linear map
\[
\frac{V^2 }{p}:W\to W^{(p^2 )}.
\]
As in the previous section, $\sigma _0$ and $\sigma _1$ can only coincide at finitely many points, since the first map has degree $2g_X-2$ and the second has degree $p^2 (2g_X-2)$. This completes the proof of Proposition \ref{prop:finites}.

\section{Degree estimates}\label{sec:finiteness}

In this section we complete the proof of part (1) of Theorem \ref{theorem:main}. Recall that in Proposition \ref{prop:zar}, it was proved that the Zariski closure of the image of $A^{[2]}(R)$ in $J^1 A_k $ is contained in a union of the form $\cup _{i=1}^n \cup _{j=1}^{m_i} (\mathcal{M} +V_i +\gamma _{ij} )$, where $V_i$ is a vector subspace of $\Ker (\pi : J^1 A_k \to A_k )$ of codimension at least 2, $\gamma _{ij}$ is a point of $J^1 A_k /\mathcal{M}$, and 
\begin{align*}
n & \leq \sum _{e=1}^{g_A -2}\frac{p^{2\binom{g_A }{e}}-1}{p-1}\\
m_i & \leq p^{g_A} .
\end{align*}
Having proved finiteness of $\red (X(R)\cap A^{[2]})$ in Proposition \ref{prop:finite} and \ref{prop:finites}, to complete the proof of part (1) of Theorem \ref{theorem:main} it will hence be enough to rpove the following proposition.
\begin{proposition}\label{prop:deg_buium}
If $V_0$ is a vector subgroup of $\Ker (\pi :J^1 A_k \to A_k )$ of codimension at least 2, and $\gamma $ is a point of $J^1 A_k $, then
\[
\# J^1 X_k \cap (\mathcal{M}+V_0 +\gamma )(k)\leq (\deg _H X_k +p(2g_X -2))\binom{g_A -2}{g_A}(3p^2 )^{g_A}\deg _H A_k .
\]
\end{proposition}
We do this by following Buium's argument in the case of the Manin--Mumford conjecture \cite{buium:96}. First, following Buium, we construct a compactification of $J^1 A_k $ as follows. As explained in \cite{buium:96}, for any projective variety $Y/R$, $J^1 Y_k \to Y_k$ has the structure of a $F^* T_{Y_k |k}$ torsor. The class of this torsor corresponds to an extension
\[
0\to F^* T_{Y_k |k}\to E_Y \to \mathcal{O}_{Y_k }\to 0
\]
of vector bundles on $Y_k $. We define $\overline{J^1 Y_k }=\mathbb{P}(E_Y )$ to be the projectivisation of $E_Y \to Y_k $. This is a projective bundle over $Y_k$, with an open immersion
\[
J^1 Y_k \hookrightarrow \overline{J^1 Y_k }.
\]
For $J^1 X_k$, we similarly construct a compactification $\overline{J^1 X_k }$ as in \cite{buium:96}.

By Bezout's theorem, we have
\begin{align*}
\# \red (X(R)\cap A^{[2]})\leq \deg _{L} (J^1 (X_k )\cap W) \\
\leq \deg _{L} (\overline{J^1 (X_k )})\cdot \deg _{L}(W),
\end{align*}
where $W$ is the Zariski closure of the image of $A^{[2]}(R)$ in $\overline{J^1 A_k }$, and $L$ is any very ample line bundle on $\overline{J^1 A_k }$. 

As in loc. cit., if $H$ is an ample line bundle on $A_k$, then $L:=3\pi ^* H+\mathcal{O}_{\mathbb{P}(E_A)}$ is a very ample line bundle on $\overline{J^1 A_k }$, where $\pi :\overline{J^1 A_k }\to A_k$ is the projection. We have the following degree estimate (see \cite[proof of Theorem 1.11]{buium:96} for a similar argument).
\begin{lemma}\label{lemma:easy_bd}
\[
\deg _L \overline{J^1 X_k} =2\deg _H X_k +p(2g_X -2).
\]
\end{lemma}
\begin{proof}
We have 
\[
\deg _L \overline{J^1 (X_k )} = (\mathcal{O}_{\mathbb{P}(E_X )}\cdot \mathcal{O}_{\mathbb{P}(E_X )})_{\mathbb{P}(E_X )} +2 X_k \cdot H. 
\]
The identity $(\mathcal{O}_{\mathbb{P}(E_X )}\cdot \mathcal{O}_{\mathbb{P}(E_X )})_{\mathbb{P}(E_X )} =p(2g_X -2)$ is from \cite{buium:96}.
\end{proof}
To complete the proof of Proposition \ref{prop:deg_buium}, and hence of part (1) of Theorem \ref{theorem:main}, it is enough to prove the following degree bound.
\begin{proposition}\label{prop:deg_bd}
If $V_0 $ is a vector space of dimension $e$, then
\[
\deg _L (\overline{\mathcal{M}+V_0 +\gamma })\leq \binom{g_A+e}{e}(3p^2 )^{g_A} \cdot \deg _H A_k .
\]
\end{proposition}

To describe the compactification and compute its degree, we use the fact that the addition map
\[
\mathcal{M}\times N\to J^1 A_k 
\]
is surjective. Here $N:=\Ker (J^1 A_k \to A_k )$. The pullback of the torsor $J^1 A_k $ along $\mathcal{M}\to A_k $ splits via the natural map $\mathcal{M}\to \mathcal{M}\times _{A_k }J^1 A_k $ induced by the inclusion $\mathcal{M}\to J^1 A_k $. Following the recipe above, we may construct a compactification $\overline{\mathcal{M}\times _{A_k }J^1 A_k }$, which is then isomorphic to the natural compactification $\mathcal{M}\times \overline{N}$, where $\overline{N}=\mathbb{P}(N\oplus k)$, and fits in a commutative diagram
\[
\begin{tikzcd}
\mathcal{M}\times _{A_k }J^1 A_k  \arrow[d] \arrow[r] & J^1 A_k \arrow[d] \\
\overline{\mathcal{M}\times _{A_k }J^1 A_k } \arrow[r]           & \overline{J^1 A_k }.
\end{tikzcd}
\]
The subvariety $\mathcal{M}+V_0 +\gamma $ will be the image of $\mathcal{M}\times (V_0 +\gamma ')$ for some $\gamma ' \in N$, hence the pre-image of the closure of $\mathcal{M}+V_0 +\gamma $ is $\mathcal{M}\times (\overline{V_0 +\gamma '})$. 

Hence the Zariski closure of $A^{[2]}(R)$ in $\overline{J^1 A_k }$ is the image, under the finite morphism.
\[
\rho :\mathcal{M}\times \overline{N}\to \overline{J^1 A_k },
\]
of a finite union of subvarieties of the form $\mathcal{M}\times (\overline{V}_0 +\gamma )$.

Let $L_0 $ denote the pullback of $H$ to $\mathcal{M}$. Then we have the following estimate, where $\pi _1 $ and $\pi _2 $ are the obvious projections and $H_0$ is the class of $\mathcal{O}(1)$ on $\overline{N}$.
\begin{lemma}
\[
\deg _L (\overline{\mathcal{M}+V_0+\gamma })\leq \deg _{\pi _1 ^* L_0 +\pi _2 ^* H_0 }(\mathcal{M}\times (\overline{V_0 +\gamma '})).
\]
\end{lemma}
\begin{proof}
Note that $\pi _1 ^* L_0 +\pi _2 ^* H_0 =\rho ^* L$. By Fulton's projection formula \cite[2.5(c)]{fulton}, we have
\begin{align*}
\deg _{\rho ^* L}(\mathcal{M}\times (\overline{V_0 +\gamma '}) & =\deg (\rho )\cdot \deg _L (\rho (\mathcal{M}\times (\overline{V_0 +\gamma '}))) \\
& = \deg (\rho )\cdot \deg _L (\overline{\mathcal{M}+V_0 +\gamma }).
\end{align*}
\end{proof}

Hence Proposition \ref{prop:deg_bd} is implied by the following lemma.
\begin{lemma}
\[
\deg _{\pi _1 ^* L_0 +\pi _2 ^* H_0 }(\mathcal{M}\times (\overline{V_0 +\gamma '}))\leq \binom{g_A +e}{e}(3p^2 )^{g_A} \deg _H (A_k ).
\]
\end{lemma}
\begin{proof}
We have 
\begin{align*}
\deg _{\pi _1 ^* L_0 +\pi _2 ^* H_0 }(\mathcal{M}\times (\overline{V_0 +\gamma '}))& := (\pi _1 ^* L_0 +\pi _2 ^* H_0 )^{g_A+e}\cdot (\mathcal{M}\times \overline{V_0 +\gamma '}) \\
& = \sum _{i=0}^{g_A +e}\binom{g_A +e }{i}(L_0 ^i \cdot \mathcal{M})\cdot (H_0 ^{g_A +e-i}\cdot \overline{V_0 +\gamma' }) \\
 & =\binom{g_A+e}{e}\deg _{L_0 }(\mathcal{M})\cdot \deg _{H_0 } (\overline{V_0 +\gamma '}) \\
& =\binom{g_A+e}{e}\deg _{L_0 }(\mathcal{M})\\
\end{align*}
Note that $\deg _{H_0 }(\overline{V_0 +\gamma '} )$ is simply the degree of $\mathbb{P}^e$ inside $\mathbb{P}^g$.
The estimate 
\[
\deg _{L_0 }(\mathcal{M})\leq (3p^2 )^{g_A} \deg _H (A_k ) 
\]
is as in \cite{buium:96}.
\end{proof}

%
\section{Bounding the number of points on a residue disc}\label{sec:disk}
In this section, we bound the number of points $X(R)\cap A^{[2]}(R)$ on a fixed residue disc, assuming property (*). This will follow from a bound on the number of zeroes of polynomials in $p$-adic abelian integrals on the curve $X$ which is a special case of a theorem of Betts \cite{betts:weight}. Recall \cite{coleman:chabauty} that one may bound the number of zeroes of a $p$-adic abelian integrals on a residue disc in terms of the genus of the curve. The quantities we need to bound are determinants of $p$-adic abelian integrals, which are substantially more difficult to handle. To approach this, we take the apparently drastic step of treating an polynomial in $p$-adic abelian integrals as an arbitrary linear combination of $p$-adic iterated integrals \cite{besser}. Recall that an iterated integral $\int \omega _0 \ldots \omega _d $ is a Coleman function satisfying 
\[
d\int \omega _0 \ldots \omega _d =\omega _0 \int \omega _1 \ldots \omega _d .
\]
Polynomials in abelian integrals arise as linear combinations of iterated integrals via the shuffle product formula
\[
\sum _{\sigma \in S_n }\int \omega _{\sigma (0)} \ldots \omega _{\sigma (d)} =\prod _{i=0}^d \int \omega _i .
\]
%
Let $K_0$ be a finite extension of $\Q _p $. In Bett's terminology, a differential operator $\sum _{i=0}^n g_i (t)\frac{d^i}{dt^i}\mathfrak{D}\in K_0 [\! [t]\! ][\frac{d}{dt}]$ is \textit{PD-nice} if the $g_i$ are in the ring $\mathcal{O}_{K_0} [\! [t]\! ]^{\mathrm{PD}}$ of power series with divided powers (i.e. the $g=\sum c_n t^n\in K_0 [\! [t]\! ]$ with $\val (c_n )\geq -\val (n!)$) and $g_n $ is in $\mathcal{O}_{K_0} [\! [t]\! ]^{\mathrm{PD},\times }$.

\begin{theorem}\label{thm:betts}[Betts, \cite{betts:weight},Proposition 5.2.3]
Let $\mathfrak{D}\in K_0 [\! [t]\! ][\frac{dt}{t}]$ be a PD-nice differential operator of order $N$. Suppose that $f\in K_0 [\! [t]\! ]$ is a non-zero power series such that $\mathfrak{D}(f)=0$. Then $f$ converges on the open disc of radius $p^{-1/(p-1)}$, and for every $\lambda > \frac{1}{p-1}$, the number of $\mathbb{C}_p $ zeroes of $f$ in the closed disc of radius $p^{-\lambda }$ is at most
\[
\left( 1+\frac{1}{(\lambda -\frac{1}{p-1})\log (p)}\right) \cdot (N-1).
\]
\end{theorem}
It follows from a result of Betts that a polynomial of degree $d$ in abelian Coleman integrals is a zero of PD-nice differential operator of order at most $(4g_X )^d \prod _{i=1}^d \binom{g_X +i+1}{g_X +1 }$ (apply \cite[Proposition 5.3.1 and Example 1.2.4]{betts:weight}
%
%
Hence we obtain the following bound.
\begin{corollary}\label{cor:betts}
The number of zeroes of a polynomial of degree $d$ in $p$-adic abelian integrals on $X$ on a closed disc of radius $p^{-1/\lambda }$ is at most
\[
	\left( 1+\frac{1}{(\lambda -\frac{1}{p-1})\log (p)}\right)\cdot (4g_X )^d\cdot\prod_{i=1}^{d-1}\binom{i+g_X+1}{g_X +1}.
\]
\end{corollary}
To apply Betts's bound, we need to construct polynomials in Coleman integrals vanishing on $X(\overline{\Q }_p )\cap A^{[2]}$. Our construction will come from the study of \textit{low rank} points as in \cite{dogra}. For an $m$-dimensional vector space $V$, $n>0$ and $r\leq \min \{ m,n \}$ let $D_{r,n}(V)\subset V^n $ denote the subvariety of elements of rank $\leq r$. Recall that this is a subvariety of codimension $(n-r)(m-r)$ in $V^n$. Furthermore, fixing a basis of $V$, under the isomorphism $\Mat _{m,n} \simeq V^n$, $D_{r,n}(V)$ is defined by the vanishing of all $(r+1)\times (r+1)$ minors.

Low rank points are related to $A^{[2]}$ via the isogeny decomposition of our abelian variety $A$:
\[
A_{\overline{K}} \sim \prod _{i=1}^n A_i ^{n_i}
\]
where the $A_i $ are pairwise non-isogenous and simple. Proper abelian subvarieties of $A_i ^{n_i}$ are necessarily images of maps of the form
\[
\psi _i :A_i ^{m_i }\to A_i ^{n_i }
\]
where $\psi _i \in \Mat _{m_i ,n_i }(\End (A_i ))$ and for some $i$, $m_i <n_i $. If $\End (A_i )=\Z $, and $m_i <\min \{ n_i ,\dim (A_i ) \}$, then 
\[
\log _{A_i ^{n_i }}(\psi _i (A_i ^{m_i })(\overline{\Q }_p ))\subset D_{m_i ,n_i }(\Lie (A)_{\overline{\Q }_p }).
\]
In particular by the above there is a Coleman function of weight $m_i +1$ on which the image of $\psi _i $ vanishes for \textit{all} such $\psi _i $ (i.e. independent of the matrix of endomorphisms).

If $\End (A_i )\neq \Z $, a similar construction can be carried out as follows.
\begin{lemma}\label{lemma:pol1}
Let $A$ be a simple abelian variety over a field $K$ of characteristic zero. Let $D\subset \End _K (\Lie (A))$ be the $K$-algebra generated by $\End (A)$. Suppose $r\leq \dim (A)/\dim (D)$. Then there is a nonzero polynomial of degree $\leq \dim (A)$ on $\Lie (A)^r$ (i.e. a nonzero element of $\oplus _{i=0}^{\dim (A)}\Sym ^i (\Lie (A^r)^* )\subset \mathcal{O}(\Lie (A^r))$ which vanishes on the image of $\Lie (B)$ for every proper abelian subvariety $B\subset A^r$.
\end{lemma}
\begin{proof}
Let $e$ denote the rank of $\Lie (A)$ as a left $D$-module, i.e. the least integer $n$ such that, for generic $x_1 ,\ldots ,x_n$ in $\Lie (A)$, the morphism
\[
(x_1 ,\ldots ,x_n ):D^n \to \Lie (A)
\]
is surjective. Note that $r\leq e$. For any $i<e$, and any $x_1 ,\ldots ,x_i$ in $\Lie (A)$, the $D$-module spanned by $x_1 ,\ldots ,x_i$ has $K$-dimension less than that of the $D$-span of $y_1 ,\ldots ,y_{i+1}$ for a generic $y_1 ,\ldots ,y_{i+1}\in \Lie (A)$. It follows that there are positive integers $m_1 ,\ldots ,m_r $, all at most $\dim (D)$, and $M_{ij}\in D$, for $1\leq i\leq r$ and $1\leq j\leq m_i$, such that, for a generic $x_1 ,\ldots ,x_r$ in $\Lie (A)$,
\[
\bigwedge _{1\leq i\leq r,1\leq j\leq m_i }M_{ij}\cdot x_i \neq 0,
\]
in $\wedge _K ^{\sum _{i=1}^r m_i }\Lie (A)$, but for any $x_1 ,\ldots ,x_r $ which are in the image of $\Lie (A)^{r-1}$ under a $D$-linear map,
\[
\bigwedge _{1\leq i\leq r,1\leq j\leq m_i }M_{ij}\cdot x_i =0.
\]
Taking the determinants of an appropriately chosen minor, we obtain a nonzero polynomial of function on $\Lie (A)^d$ of degree $\sum _{i=1}^r m_i$. The lemma follows from noting
\[
\sum _{i=1}^r m_i \leq r\cdot \dim (D)\leq \dim (A).
\]
\end{proof}

\begin{lemma}\label{lemma:pol2}
Let $A$ be an abelian variety, with $A_{\overline{K}}$  isogenous to $\prod _{i=1}^{n_i} A_i ^{m_i}$, with the $A_i$ simple and pairwise non-isogenous. Let $D_i \subset \End _{K_i }(\Lie (A_i ))$ be the $K$-algebra spanned by $\End (A)$. Suppose that for all $i$, $n_i \leq \dim (A_i )/\dim (D_i )$. Then there exist a set of $n_i$ nonzero polynomials in $\mathcal{O}(\Lie (A))$ of degree at most $\dim (A)$ vanishing on $\Lie (B)$ for any subabelian variety $B$ of codimension at least 2. 
\end{lemma}
\begin{proof}
It is enough to prove this in the case where $A=\prod _{i=1}^{n_i} A_i ^{n_i }$ as above. Let $B$ be a sub-abelian variety of $A$ of codimension at least $2$. Either there is at least one $A_i$ of dimension at least 2 such that $B$ does not surject onto $A_i$, or there are at least two $A_i$ of dimension $1$ such that $B$ does not dominate either of them. In the first case, via projection onto $A_i$ we obtain a polynomial of degree $\leq \dim (A_i )$ from Lemma \ref{lemma:pol1}. In the second case, $\Lie (B)$ is contained in the annihilator of $\Lie (A_i )^*$ for some $A_i$ of dimension 1.
\end{proof}

Combining Lemma \ref{lemma:pol2} and Corollary \ref{cor:betts} we obtain the following bound.
\begin{proposition}\label{prop:final_bound}
For $A$ as in Lemma \ref{lemma:pol2}, $K$ a finite extension of $\Q _p$, and $X\subset A$ a curve of genus $g_X $ over $K$ with good reduction which generates $A$, and $z\in X(k)$, with $z_0 \in X(K)$ lying above it, and $T$ an integral parameter at $z_0$, we have
\begin{align*} 
& \# A^{[2]}\cap \{ z_1 \in ]z[(\overline{\Q _p }):|T_1 (z_1 )| \leq   p^{-\lambda } \} \\
 \leq & n_A \cdot \left( 1+\frac{1}{(\lambda -\frac{1}{p-1})\log (p)}\right) \cdot (4g_X )^{g_A} \cdot\prod_{i=1}^{g_A-1}\binom{i+g_X+1}{g_X+1},
\end{align*}
where $n_A$ is the number of distinct isogeny classes of simple abelian varieties over $\overline{\Q }_p $ admitting a nonzero homomorphism to $A_{\overline{\Q }_p }$.
In particular, we obtain
\begin{align*} 
& \# X(R) \cap A^{[2]}\cap ]z[  \\ 
\leq &  n_A \cdot \# \red (X(R)\cap A^{[2]}) \cdot \left( 1+\frac{1}{(1 -\frac{1}{p-1})\log (p)}\right) \cdot (4g_A )^{\dim (A)}\cdot \prod _{i=1}^{\dim (A)} \binom{i+g_X +1}{g_X +1}
\end{align*}
\end{proposition}

\section{Bounding ramified points}\label{sec:ram}
In this section we prove part (2) of Theorem \ref{theorem:main}, and complete the proof of part (3). Recall that this is the statement that 
\[
\# \red ((X(\overline{\Q }_p )-X(R))\cap A^{[2]}) \leq (2g_X -2)\left(1+\frac{\sum _{e=1}^{g_A -2}p^{2\binom{g_A }{e}}-1}{p-1}\right).
\]
The result (and its proof) are generalisations of Coleman's theorem on ramified torsion points \cite{coleman:ramified}. We will also use Coleman's method to prove a bound on the number of ramified points in $(X(\overline{\Q }_p )-X(R))\cap A^{[2]}$ for $A$ satisfying property (*). 
%

Our main result in this section will be the following Proposition.
\begin{proposition}\label{prop:coleman1}
Let $W$ be a $2$-dimensional $\Q _p ^{\mathrm{nr}}$-subspace of $H^0 (X_{\Q _p ^{\mathrm{nr}}},\Omega )$. Let 
\[
S_W :=\cap _{\omega \in W} \{ z\in X(\overline{\Q }_p ):\int ^z _b \omega =0 \} .
\]
Then
\begin{enumerate}
\item 
$S_W$ contains only finitely many ramified points, and all such points have ramification degree at most $2g_X -2$. 
\item If $z$ is a ramified point in $S_W $, then $\red (z)$ is a point where all $\omega $ in $(W\cap H^0 (X,\Omega ))\otimes _R \overline{\F }_p \subset H^0 (X_{\overline{\F }_p },\Omega )$ vanish. 
\end{enumerate}
\end{proposition}
Note that part (2) of Proposition \ref{prop:coleman1} implies part (2) of Theorem \ref{theorem:main}. Indeed, since the endomorphism algebra of $A_{\overline{\Q }_p}$ is defined over $K$, every vector subspace of $H^0 (A_{\overline{\Q }_p },\Omega )$ which arises as the kernel of pullback to $H^0 (B,\Omega )$ for some sub-abelian variety $B<A_{\overline{\Q }_p }$ is in fact defined over $K$. Furthermore, by Lemma \ref{lemma:dubious}, the number of subspaces of $H^0 (X_{\overline{\F }_p }, \Omega )$ which arise as $(W\cap H^0 (X,\Omega ))\otimes _R \overline{\F }_p \subset H^0 (X_{\overline{\F }_p },\Omega )$ for $W\subset H^0 (X_{\overline{\Q }_p },\Omega )$ the differentials vanishing along some sub-abelian variety $B<A$ is at most $1+\frac{\sum _{e=1}^{g_A -2}p^{2\binom{g_A }{e}}-1}{p-1}$. For each such subspace, there are at most $2g_X -2$ points where all differentials vanish.
\subsection{Coleman expansions of abelian integrals}
The proof of Proposition \ref{prop:coleman1} uses the \textit{Coleman expansion} of an abelian integral (in the terminology of \cite{DP25}) which was developed in \cite{coleman:ramified}. We recall the construction briefly. Given a nonzero $\omega \in H^1 _{\dR} (X_K /K)$, we define the \textit{valuation} of $\omega $ to be the largest integer $n$ such that $\omega \in p^n H^1 _{\dR}(X/R)$. Let $V$ denote the Verschiebung operator on $H^1 _{\dR}(X/R)$. We have an infinite strictly increasing sequence $(n_i )$ of non-negative integers, defined by the property that $n_i$ is the $i$th non-negative integer $m$ such that $\val (V^m (\omega ))=\val (V^{m+1}(\omega ))$. If the reduction modulo $p$ of $\omega $ in $H^1 _{\dR}(X/\overline{\F }_p )$ is nonzero, then for all $i$, $V^{n_i+1}(\omega )/p^{n_i-i}$ is an element of $H^1 _{\dR}(X/R)$ whose reduction modulo $p$ is nonzero, and whose image in $H^1 _{\dR}(X/\overline{\F }_p )$ lies in the subspace $H^0 (X_{\overline{\F }_p },\Omega )$. For $z\in X(\overline{\F} _p )$, we define $k_i (\omega ,z)$ to be the valuation of the corresponding differential at $z$. We recall the following result, which can easily be deduced from Coleman's work \cite{coleman:ramified}. 
\begin{lemma}[\cite{DP25}, Lemma 8]\label{lemma:DP}
Suppose $\lambda $ is a (negative) slope of the Newton polygon of $\int ^z _b \omega \in K[\! [T]\! ]$. If $\lambda < \frac{1}{2g_X-2}$, then one of the following holds.
\begin{enumerate}
\item 
\begin{equation}\label{eqn:slopes}
\lambda =\frac{1}{p^{n_{i+1} (\omega )}k_{i+1}(\omega ,\overline{z})-p^{n_i (\omega )}k_{i}(\omega ,\overline{z})}
\end{equation}
for some $i\geq 0$.
\item $\val (\int _b ^{z_0 }\omega )\geq 0$, and $\lambda =\frac{1}{k_N (\omega ,\overline{z} )} p^{n_N (\omega ,\overline{z})}$, where $N=N(\omega ,z_0 )$ is $1-\val \int _b ^{z_0 }\omega $.
\end{enumerate}
\end{lemma}

\subsection{The outer provinces}
To prove part 1 of Proposition \ref{prop:coleman1}, it will be enough to prove that there does not exist a $\lambda <\frac{1}{2g_X-2}$ which is of the form $\frac{1}{p^{n_{i+1} (\omega )}k_{i+1}(\omega ,\overline{z})-p^{n_i (\omega )}k_{i}(\omega ,\overline{z})}$ or $p^{-N(\omega ,z_0 )}$ for all $\omega $ in $W$. Suppose such an $\omega $ did exist. By our assumptions on $p$ and $g$, there cannot be $\omega ,\eta$ and integers $i,j$ such that 
\[
\frac{1}{k_i (\omega ,\overline{z} )p^{n_i (\omega ,\overline{z} )}}=\frac{1}{k_{j+1} (\eta ,\overline{z} )p^{n_{j+1} (\eta ,\overline{z} )}-k_{j} (\eta ,\overline{z} )p^{n_{j} (\eta ,\overline{z} )}}.
\]
Hence it must be the case either that, for all $\omega $ in $W$, $\lambda =\frac{1}{k_N (\omega ,\overline{z} )} p^{n_N (\omega ,\overline{z})}$, or for all $\omega $
\[
\lambda =\frac{1}{p^{n_{i+1} (\omega )}k_{i+1}(\omega ,\overline{z})-p^{n_i (\omega )}k_{i}(\omega ,\overline{z})}
\]
for some $i$. Suppose we are in the second case, and fix $\omega ,\eta \in W$. Then, since $k_i $ must be at most $2g_X-2$, we deduce there are some $i,j$ such that
$
k_i (\omega ,\overline{z})=k_j (\eta ,\overline{z})$ and $n_i (\omega ,\overline{z})=n_j (\eta ,\overline{z})$.

We now apply the following lemma, which is a mild generalisation of a lemma of Coleman.
\begin{lemma}\label{lemma:outer}
Let $M$ be an $R$-submodule of $H^0 (X,\Omega )$ and $n\geq 0 $. Suppose, for all $\omega \in M$, $\overline{V^n (\omega )}$ lies in $H^0 (X_{\overline{\F _p }},\Omega )$. Then 
\[
\{ \overline{V^n (\omega )}:\omega \in M \}
\]
contains a $\rk M$-dimensional subspace of $H^0 (X_{\overline{\F }_p },\Omega )$.
\end{lemma}
\begin{proof}
We may $\sigma $-linearise the morphism to get a homomorphism of $R$-modules
\[
\widetilde{V}^n :M\otimes _{R,\sigma }R\to H^1 _{\dR}(X/R)
\]
with image equal to that of $V^n$. Then we may choose bases of the left and right hand sides so that the entries of the matrix defining this map are zero away from the diagonal, and powers of $p$ on the diagonal.
\end{proof}
In particular, we deduce that there cannot be an integer $n$ such that for all $\omega ,\eta $ in $W$ there exists an $i,j$ with $n_i (\omega )=n_j (\eta )=n$ and $k_i (\omega )=k_j (\eta )$. This completes the proof of part (1) of Proposition \ref{prop:coleman1}.

\subsection{The inner provinces}
We now give the proof of part 2 of Proposition \ref{prop:coleman1}, and complete the proof of part (3) of Theorem \ref{theorem:main}. 
By the previous section, if  there is a zero of slope of valuation $\geq \frac{1}{2g_X-2}$, then $\val (\int ^{z_0 }_{b_0 }\omega )>0$, and the endpoint of the portion of the Newton polygon with slopes $\geq \frac{1}{2g_X-2}$ occurs at the first zero of the power series expansion of $\omega $ at $\overline{z}$. Hence the set of $\overline{z}\in X(\overline{\F }_p )$ containing a ramified point at which $\int _b \omega $ vanishes for all $\omega $ in $W$ is contained in the set of $\overline{z}$ in $X(\overline{\F }_p )$ at which every $\omega $ in $(W\cap H^0 (X,\Omega ))\otimes _R \overline{\F }_p \subset H^0 (X_{\overline{\F }_p },\Omega )$ vanishes. This completes the proof of Proposition \ref{prop:coleman1}.

The ramification bound afforded by Proposition \ref{prop:coleman1} enables us to apply Betts's Chabauty--Coleman bound to prove the following result, which completes the proof of part (3) of Theorem \ref{theorem:main}.

\begin{corollary}\label{cor:ramm}
Suppose $A$ satisfies (*).
Then the number of points of $(X(\overline{\Q }_p )-X(\Q _p ^{\mathrm{nr}}))\cap A^{[2]}$ is at most 
\[
(2g_X -2)(1+\frac{\sum _{e=1}^{g_A-2}p^{2\binom{g_A }{e}}-1}{p-1})
\left( 1+\frac{1}{(\frac{1}{2g_X} -\frac{1}{p-1})\log (p)}\right) \cdot (4g_A )^{\dim (A)}\cdot \prod _{i=1}^{\dim (A)} \binom{i+g_X +1}{g_X +1}
\]
\end{corollary}
\begin{proof}
By Proposition \ref{prop:coleman1}, the set $(X(\overline{\Q }_p )-X(\Q _p ^{\mathrm{nr}}))\cap A^{[2]}$ is contained in at most $(2g_X-2)(1+\frac{\sum _{e=1}^{g_A-2}p^{2\binom{g_A }{e}}-1}{p-1})$ residue discs, each of radius at most $p^{-1/2g_X}$. Hence the corollary follows from Theorem \ref{thm:betts}.
\end{proof}

\section{Bounding the degree}\label{sec:deg}
In this section, we prove Corollary \ref{cor:deg}, following similar arguments in \cite{DP25} (in the case of Mordell--Lang) and \cite{approaches} (in the case of curves in tori). It will be enough to prove the following.

\begin{proposition}\label{prop:gal}
There is a positive integer $n$ such that 
\[
S_n :=\{ (\sigma _1 (x),\ldots ,\sigma _n (x))\in X(\overline{\Q })^n :x\in X(\overline{\Q })\cap A^{[2]}, \sigma _i \in \Gal (\overline{\Q }|K) \}
\]
is not Zariski dense in $X^n$.
\end{proposition}
To prove Proposition \ref{prop:gal}, fix an extension $L$ of $K$ over which $\End (A)_{\overline{K}}$ is defined, and a prime $\mathfrak{p}$ of $K$ at which $X$ and $A$ have good reduction, and which is unramified over $\Q$. It will be enough to prove the proposition with $K$ replaced by $L$. Note that, by Theorem \ref{theorem:main}, there are a finite set of closed discs $D$ in $X^{n,\an }$ such that $S_n$ is contained in the union of the $D$. Hence it will be enough to prove the proposition for $S_n \cap D$. We follow the strategy of \cite{approaches}. Namely, we show that $S_n \cap D$ is contained inside the zeroes of a set of polynomials in Coleman integrals, and use functional transcendence (Ax's theorem) to show these common zeroes are not Zariski dense in $X^n$.

Note that, since $\End (A)_{\overline{K}}$ is defined over $L$, if $B\subset A$ is an abelian subvariety, and $P\in A(\overline{K})$ is torsion in $A/B$, then $\sigma (P)$ is torsion in $A/B$ for all $\sigma \in \Gal (\overline{K}|L)$. Now let 
\[
\log _{A^n} :D\to \Lie (A)_n
\]
be the abelian logarithm (as in section \ref{sec:ram}) restricted to the disc $D$. If $\log _A (P)$ is in $\Lie (B)$, then $\log _A (\sigma (P))$ is in $\Lie (B)$ for all $\sigma $ in $\Gal (\overline{K}|L)$. We deduce that, for all $P\in S_n $, $\log _{A^n }(P)$ lies in $\Lie (B)^n$, and in particular in the subvariety $D_r (\Lie (A)^n )$ of rank $\leq r:=\dim (B)$ tuples in $\Lie (A)^n \simeq \Mat _{n,g}$. Recall that, if $r\leq \min \{n,g \}$, this has codimension $(g-r)(n-r)$ in $\Lie (A)^n $.

We aim to show that $\log _{A^n }^{-1}(D_r (\Lie (A)^n ))$ is not Zariski dense in $X^n$. Since $\log _{A^n }$ is a morphism from a Noetherian rigid analytic space, it will be enough to prove that, at each point $z\in \log _{A^n }^{-1}(D_r (\Lie (A)^n ))$, the formal completion of  $\log _{A^n }^{-1}(D_r (\Lie (A)^n ))$ at $z$ is not Zariski dense in $X^n$. For this we make use of the Ax--Schanuel theorem for abelian varieties.
\begin{theorem}[Ax, \cite{ax1972some}]\label{thm:ax}
Let $G$ be an abelian variety over a field $K$ of characteristic zero, $V\subset G$ is an irreducible subvariety, and $W\subset \Lie (G)$ an irreducible subvariety. Suppose $V$ and $W$ contain the origin. Let 
\[
\widehat{\log }:\widehat{G}\to \widehat{\Lie }(G)
\]
be the logarithm isomorphism from the formal group of $G$ to that of its Lie algebra. Let $\widehat{W}$ and $\widehat{G}$ denote the formal completions of $W$ and $G$ at the origins of $\Lie (G)$ and $G$ respectively. Let $Z$ be an irreducible component of $\widehat{V} \cap \widehat{\log }^{-1}(\widehat{W})$. If
\[
\dim Z > \dim (V)-\codim _{\Lie (G)}W,
\]
then $Z$ is contained in the formal completion of an abelian subvariety of $G$.
\end{theorem}
We apply Ax's theorem with $G=A^n$, $V$ taken to be the translation of $X^n$ by the image of $z$, $W$ taken to be the translation of $D_r (\Lie (A)^n )$ by the image of $\log _{A^n }(z)$. In this way, the formal completion of $\log _{A^n }^{-1}(D_r (\Lie (A)^n ))$ is identified with $\widehat{V}\cap \widehat{\log }^{-1}(\widehat{W})$. Since $X^n $ generates $A$, Ax--Schanuel implies that $\log _{A^n}^{-1}(D_r (\Lie (A)^n )$ is not Zariski dense.

\bibliography{bib_ZP}
\bibliographystyle{alpha}

\end{document}